\documentclass[10pt,a4paper]{article}

\usepackage{latexsym}
\usepackage{verbatim}
\usepackage{amsthm}
\usepackage{amsfonts}
\usepackage{amsmath}
\usepackage{amssymb}
\usepackage{amscd}

\numberwithin{equation}{section}

\newtheorem{prop}{Proposition}[section]
\newtheorem{theorem}[prop]{Theorem}
\newtheorem{lemma}[prop]{Lemma}
\newtheorem{corollary}[prop]{Corollary}
\newtheorem{remark}[prop]{Remark}

\newtheorem{definition}[prop]{Definition}

\def\begeq{\begin{equation}}
\def\endeq{\end{equation}}

\def\<{\langle}
\def\>{\rangle}
\def\({\left(}
\def\){\right)}

\def\p{\partial}

\def\om{\omega}
\def\Om{\Omega}
\def\we{\wedge}
\def\ep{\epsilon}
\def\dc{dd^c}
\def\MAE{Monge-Amp\`ere equation}

\DeclareMathOperator{\sh}{\mathcal {SH}}

\DeclareMathOperator{\psh}{\mathcal {PSH}}

 \DeclareMathOperator{\la}{\lambda}

\begin{document}

\title{A priori estimates for the complex Hessian
equations}
\author{S\l awomir Dinew and S\l awomir Ko\l odziej}
\date{}
\maketitle
\begin{abstract}
 We prove some $L^{\infty }$ a priori estimates as well as
existence and stability theorems for the weak solutions of the
complex Hessian equations in domains of $\mathbb C^n$ and on
compact K\"ahler manifolds. We also show optimal $L^p$
integrability for $m$-subharmonic functions with compact
singularities, thus partially confirming a conjecture of B\l ocki.
Finally we obtain a local regularity result for $W^{2,p}$
solutions of the real and complex Hessian equations under suitable
regularity assumptions on the right hand side. In the real case
the method of this proof improves a result of Urbas.
\end{abstract}
\section*{Introduction}
\bigskip
\it Hessian equations. \rm  Let $ \la = (\la _1, \la _2 , ... ,
\la _n)$ be the set of eigenvalues of a Hermitian $n\times n$
matrix $A$. By $S_m (A)$ denote the m-th elementary symmetric
function of $\la $:
$$
S_m (A)=\sum _{0<j_1 < ... < j_m \leq n }\la _{j_1}\la _{j_2} 
... \la _{j_m}.
$$
If $A$ is the complex Hessian of a real valued $C^2$ function $u$
defined in $\Om \subset\mathbb C^n$ then we have a pointwise
defined function
$$
\sigma_m (u_{z_j\bar z_k})(z) = S_m\big((u_{z_j\bar z_k}(z))\big).
$$
In terms of differential forms, with $d=\p +\bar{\p }, \  d^c
=i(\bar{\p } -\p )$ and $\beta = dd^c ||z||^2$ this function
satisfies
$$ (dd^c u)^m \we \beta ^{n-m} =  \frac{m!(n-m)!}{n!}\sigma_m
(u_{z_j\bar z_k})\beta ^n. \rm
$$
We call a $C^2$ function $u:\Om \to \mathbb C^n$ \it m-subharmonic
(m-sh) \rm if the forms
$$
(dd^c u)^k \we \beta ^{n-k}
$$
are positive for $k=1, ... , m$ (in particular $u$ is
subharmonic). If $u$ is subharmonic but not smooth then one can
definine $m$-sh function via inequalities for currents (see
definitions in Preliminaries).

As shown by B\l ocki in \cite{Bl} $m$-sh functions are the right
class of admissible solutions to the complex Hessian equation
\begin{equation}\label{hes}
(dd^c u)^m \we \beta ^{n-m} = f\beta
^n \end{equation}
 for given nonnegative function $f$. Observe that for $m=1$ this
 is the Poisson equation and for $m=n$ the complex \MAE .

In analogy to the above one can define \it m-subharmonic functions
with respect to a K\"ahler form $\om $ ($m-\om$-sh) \rm and the
corresponding Hessian equation just replacing $\beta$ with $\om$
in the preceeding definitions. This definition can also be
extended to subharmonic functions. Then one can consider such
functions on K\"ahler manifolds.

Since on compact K\"ahler manifolds the sets of $m-\om$-sh
functions are trivial we define in this case $\om-m$-subharmonic
($\om-m$-sh) functions requiring that
$$
(dd^c u +\om )^k \we \om^{n-k} \geq 0 ,  \ \ k=1, ... , m.
$$
and consider the Hessian equation on a compact K\"ahler manifold
$X$, as in \cite{Hou} and \cite{HMW}:
\begin{equation}\label{hes1}
(dd^c u +\om )^m \we \om^{n-m}=f\om ^n ,  \ \ \int _X f\om ^n
=\int _X  \om ^n .
\end{equation}
Solving the equation we look for $u$ which is $\om-m$-sh. The
normalization of $f$ is necessary because of the Stokes theorem
and the K\"ahler condition $d\om =0.$

\bigskip
\it Background. \rm The real Hessian equation was studied in many
papers, for example in \cite{CNS}, \cite{ITW}, \cite{Kr},
\cite{Tr}, \cite{TW},
\cite{La}, \cite{CW}, \cite{U}; to mention only a few. %(and we apologize for missing several important references).
In particular the Dirichlet problem is solvable for smooth and
strictly positive right hand side under natural convexity
assumptions on the boundary of the considered domain (\cite{CNS}).
This result is the starting point of study of degenerate Hessian
equations (\cite{ITW}) and regularity of weak solutions
(\cite{U}). Furthermore a non linear potential theory has been
developed (\cite{TW}, \cite{La}). We refer to \cite{Wa2} for
state-of-the-art survey of the real Hessian equation theory. It is
somewhat interesting  that  real and complex theories are very
different, and attempts to use directly the "real" methods  to the
complex Hessian equation often fail. See \cite{Bl2} or \cite{Bl3}
for a detailed study of those discrepancies.

The complex Hessian equation (\ref{hes}) in domains of $\mathbb
C^n$ was first considered by S.Y. Li \cite{Li}. His main result
says that if $\Om$ is smoothly bounded and $(m-1)$-pseudoconvex
(that means that $S_{j},\ j=1,\cdots, m-1 $ applied to the Levi
form of $\p\Om$ are positive on the complex tangent to $\p\Om$)
then, for smooth boundary data and for smooth, positive right hand
side there exists a unique smooth solution of the Dirichlet
problem for the Hessian equation. The proof is in the spirit of
the one in \cite{CNS}.

In \cite{Bl} B\l ocki considered also weak solutions of the
equation, for possibly degenarate right hand side, introducing
some elements of potential theory for m-sh functions based on
positivity of currents which are used in the definition. He proved
that the m-sh function $u$ is maximal in this class if and only if
$$
(dd^c u)^m \we \beta ^{n-m}=0. $$ Furthermore he  described the
maximal domain of definition of the Hessian operator.

As for the equation on compact K\"ahler manifolds (\ref{hes1}) Hou
\cite{Hou} has shown that the solutions, for smooth positive $f$,
exist under the assumption that the metric has nonnegative
holomorphic bisectional curvature. Similar results were
independently obtained in  \cite{Kok}, \cite{J}.  Despite the
further efforts \cite{HMW} the general case is still open.

\bigskip
\it New results. \rm The $m$-subharmonic functions for  $m<n$ are
much more difficult to handle than the plurisubharmonic ones
($m=n$). They lack a nice geometric description by the mean value
property along planes, there is no invariance of the family under
holomorphic mappings, and so forth. The cones of
m-$\om$-sh functions are even worse - they are not invariant under
translations. Despite that the pluripotential theory methods
developed in \cite{BT2}, \cite{K0}, \cite{K3}, \cite{KI} for the
\MAE\ can be adapted to the Hessian equations. The crucial
estimate between volume and capacity in Proposition
\ref{volcapestimate} allowed us to prove a sharp integrability
statement (conjectured in a stronger form in \cite{Bl}):
m-subharmonic functions, $m<n$,  belong to $L^q $ for any
$q<\frac{mn}{n-m},$ if their level sets are relatively compact in
the domain where they are defined. For a plurisubharmonic function
$u$ much stronger statement is true: $\exp (-au)$ is locally
integrable for some  $c>0$. This  accounts for the difference in
statements of $L^{\infty }$ estimates for the Hessian equations
and the \MAE . We show a priori bounds $L^{\infty }$  for the
solutions of
\begin{equation}\label{hes2}
(dd^c u)^m \we \beta ^{n-m} = f\om ^n \end{equation} (with
continuous boundary data) and those of (\ref{hes1}) with $f$
belonging to $L^q $, $q<\frac{n}{n-m}.$ We also get strong
stability theorems for those solutions. As a consequence one
obtains that the families of solutions corresponding to data
uniformly bounded in $L^q $ norms are equicontinuous.

The a priori estimates lead to the (continuous) solution of the
Dirichlet problem in $(m-1)$-pseudoconvex domains for nonnegative
right hand side in the same $L^q$ spaces as above (Theorem
\ref{weaksol}). The corresponding existence result is also true on
compact K\"ahler manifolds  with nonnegative holomorphic
bisectional curvature (Theorem \ref{weaksolK}). Those are the
extensions of theorems in \cite{Li} and \cite{Hou}. Finally we
prove the local regularity statement in Theorem \ref{lore} which
in the case of the \MAE\ is due to B\l ocki and Dinew \cite{BD}.
It is worth  noting that our methods applied to the  real Hessian
equations yield improvement of the regularity exponent obtained by
Urbas (\cite{U}).

 \section{Preliminaries }
We briefly recall the notions that we shall need later on. We
start with a linear algebra toolkit.

\bigskip
\it Linear algebra preliminaries. \rm Consider the set $\mathcal
M_n$ of all Hermitian symmetric $n\times n$ matrices. For a given
matrix $M\in\mathcal M_n$ let $ \la(M) = (\la _1, \la _2 , ... ,
\la _n)$ be its eigenvalues arranged in the decreasing order and
let
$$
S_k(M)=S_k(\la(M))=\sum _{0<j_1 < ... < j_m \leq n }\la _{j_1}\la _{j_2} 
... \la _{j_m}
$$
be the $k$-th elementary symmetric polynomial applied to the vector $\la(M)$.

Then one can define the positive cones $\Gamma_m$ as follows
\begin{equation}\label{ga}
\Gamma_m=\lbrace \la\in\mathbb R^n|\ S_1(\la)> 0,\ \cdots,\
S_m(\la)> 0\rbrace.
\end{equation}
Note that the definition of $\Gamma_m$ is non linear if $m>1$
hence a priori it is unclear whether these sets are indeed convex
cones. But the vectors in $\Gamma_m$, and hence the set of
matrices with corresponding eigenvalues enjoy several convexity
properties resembling the properties of positive definite
matrices, and
 in partiular the convexity of $\Gamma_m$.

Let now $V$ be a fixed positive definite Hermitian matrix and
$\la_i(V)$ be the eigenvalues of a Hermitian matrix $M$ with
respect to $V$. The we can analogously define the sets
$\Gamma_k(V)$.

Below we list the properties of these cones that will be used later on:
\begin{enumerate}
\item (Maclaurin's inequality) If $\la\in\Gamma_m$ then
$(\frac{S_j}{\binom nj})^{\frac 1j}\geq (\frac{S_i}{\binom ni})^{\frac 1i}$
for $1\leq j\leq i\leq m$;
\item (G\aa rding's inequality, \cite{Ga}) $\Gamma_m$ is a convex cone for
any $ m$ and the function $S_m^{\frac 1m}$ is concave when restriced to
$\Gamma_m$;
\item (\cite{Wa2}) Let $S_{k;i}(\la):=S_k(\la)_{\la_i=0}=
\frac{\partial S_{k+1}}{\partial \la_i}(\la)$. Then for any
$\la,\ \mu\in\Gamma_m$
$$\sum_{i=1}^n\mu_iS_{m-1;i}(\la)\geq mS_m(\mu)^{\frac1m}S_m(\la)^{\frac{m-1}{m}}.$$

\end{enumerate}

We refer to \cite{Bl} or \cite{Wa2} for further properties of these cones.

\bigskip
\it  Potential theoretic aspects of $m$-subharmonic functions. \rm
 Let us fix a relatively compact domain $\Om\in\mathbb C^n$. Let also
 $d=\partial+\bar{\partial}$ and $d^c:=i(\bar{\partial}-\partial)$ be the
 standard exterior differentiation operators. By $\beta:=dd^c||z||^2$ we denote
 the Euclidean K\"ahler form in $\mathbb C^n$.

Given a $\mathcal C^2(\Om)$ function $u$ we call it
$m-\beta$-subharmonic if for any $z\in\Om$ the Hessian matrix
$\frac{\partial^2u} {\partial z_i\partial\bar{z}_j}(z)$ has
eigenvalues forming a vector in the closure of the cone
$\Gamma_m$. Analogously if $\om$ is any other K\"ahler form in
$\Om$, $u$ is $m-\om$-subharmonic if the Hessian matrix has
eigenvalues at $z$ forming a vector in $\bar{\Gamma }_m(\om(z))$
(the latter set will depend on $z$ in general).

Since the $\om=\beta$ is the most natural case in the flat domains
we shall call $m-\beta$-subharmonic functions just $m$-subharmonic
or m-sh for short.

Observe that in the language of differential forms $u$ is $m-\om$-subharmonic
if and only if the following inequalities hold:
$$(dd^cu)^k\we\om^{n-k}\geq 0,\ k=1,\ \cdots,\ m.$$

It was obsreved by B\l ocki (\cite{Bl}) that, following the ideas
of Bedford and Taylor ({\cite{BT}, \cite{BT2}), one can relax the
smoothness requirement on $u$ and develop a non linear version of
potential theory for Hessian operators.

The relevant definitions are as follows:
\begin{definition} Let $u$ be a subharmonic function on a
domain $\Om\in\mathbb C^n$. Then $u$ is called $m$-subharmonic (m-sh for short)
if for any collection of $\mathcal C^2$-smooth m-sh functions
$v_1,\ \cdots,\ v_{m-1}$ the inequality
$$dd^cu\we dd^cv_1\we\cdots\we dd^cv_{m-1}\we\beta^{n-m}\geq 0$$
holds in the weak sense of currents. For a general K\"ahler form
$\om $ the notion of $m-\om$-subharmonic function is defined by
formally stronger condition: locally, in a neighbourhood of any
given point, there exists a decreasing to $u$ sequence of
 $\mathcal C^2$-smooth $m-\om$-sh functions $u_j$ such that for any set of
 $\mathcal C^2$-smooth $m-\om$-sh functions
$v_1,\ \cdots,\ v_{m-1}$ the inequality
$$dd^cu_j\we dd^cv_1\we\cdots\we dd^cv_{m-1}\we\om ^{n-m}\geq 0$$
is satisfied. (For $\om =\beta$ this condition is satisfied due to
part 4 of Proposition \ref{aaaa} below.)

The set of all $m-\om$-sh functions is denoted by $ \sh _m
(\om,\Om)$.
\end{definition}
\begin{remark} It is enough to test $m$-subharmonicity of $u$ against a
collection of $m$-sh
quadratic polynomials (see \cite{Bl}).
\end{remark}

Using the approximating sequence $u_j$ from the definition one can
follow the Bedford and Taylor construction from \cite{BT2} of the
wedge products of currents given by locally bounded $m-\om$-sh
functions. They are defined inductively by
$$dd^cu_1\we\cdots\we dd^cu_p\we\om^{n-m}:=dd^c (u_1\we\cdots\we dd^cu_p\we\om^{n-m}).
$$

 It can be shown (see \cite{Bl}) that
analogously to the pluripotential setting these currents are
continuous under monotone or uniform convergence of their
potentials.

%\begin{remark} Strictly speaking some of the results in \cite{Bl} are stated for
%continuous m-sh functions. But analogously to Bedford-Taylor
%theory this assumption can be relaxed by using the fine continuity
%properties of these classes of functions. \bf ?????????????  \rm
%We refer to \cite{K2}
%for a detailed study in the plurisubharmonic (i.e. $m=n$) setting
%and let the Reader to verify the analogous $m$-subharmonic
%counterparts.
%\end{remark}
 Here we list some basic facts about $m$-subharmonicity (assuming $\mathcal C^2$\ smoothness).
 \begin{prop}\label{aaaa}
 Let $\Om\subset\mathbb C^n$\ be a domain. Then
 \begin{enumerate}
 \item $\sh _1 (\om,\Om)\subset \sh _2 (\om,\Om)\subset\cdots
 \subset \sh _n (\om, \Om)$,
 %\item $\sh _m (\Om)\subset(n-m+1)-\cv(\Om)$, and the inclusion is strict if $m<n$,
 \item $\sh _m (\om, \Om)$ is a convex cone,
 \item If $u\in \sh _m (\om, \Om)$\ and $\gamma:\mathbb
 R\rightarrow\mathbb R$\ is a $\mathcal C^2$-smooth convex, increasing
 function then $\gamma\circ u\in \sh _m (\om,\Om)$,
 \item the standard regularizations $u\ast\rho_{\varepsilon}$ of a
 $m$-sh %{\rm (}$\om ${\rm )}
 function is again $m$-sh.%{\rm (}$\om ${\rm )}
 \end{enumerate}
 \end{prop}
 \begin{proof}
 The first claim is trivial.  Second claim is proved in
 \cite{Bl}, with the use of G\aa rding's inequality \cite{Ga}. Last two
 claims are more or less standard and their proofs are analogous to
 corresponding results for $\psh$\ functions.
 Observe that the last property does fail for a general K\"ahler form $\om$.
 \end{proof}

The following two theorems, known as comparison principles in
pluripotential theory, follow essentially from the same arguments
as in the case $m=n$:
\begin{theorem}\label{compprin} Let $u,\ v$ be continuous $m-\om$-sh functions in a domain
$\Om\in\mathbb C^n$. Suppose that
$\liminf_{z\rightarrow\partial\Om} (u-v)(z)\geq0$ then

$$\int_{\lbrace u<v\rbrace}(dd^c v)^{m}\we\om^{n-m}\leq\int_{\lbrace u<v\rbrace}
(dd^c u)^{m}\we\om^{n-m}.$$
\end{theorem}
\begin{theorem}\label{globcompprin}  Let $u,\ v$ be continuous $m-\om$-sh functions in a
domain $\Om\in\mathbb C^n$. Suppose that
$\liminf_{z\rightarrow\partial\Om}(u-v)(z)\geq0$ and
$(dd^cv)^{m}\we\om^{n-m}\geq( dd^cu)^{m}\we\om^{n-m}$. Then $v\leq
u$ in $\Om$.
\end{theorem}

The last result yields, in particular, uniqueness of bounded weak
solutions of the Dirichlet problem. As for the existence we have
the following fundamental existence theorem due to S. Y. Li
(\cite{Li}):
\begin{theorem}\label{liexist}
Let  $\Om$ be a smoothly bounded relatively compact domain in
$\mathbb C^n$. Suppose that $\partial \Om$ is $(m-1)$-pseudoconvex
(that means that  Levi form at any point  $p\in\partial\Om$ has
its eigenvalues in the cone $\Gamma_{m-1}$). Let  $\varphi$ be a
smooth function on $\partial \Om$ and $f$ a strictly positive and
smooth function in $\Om$. Then the Dirichlet problem
$$\begin{cases} u\in\sh_m(\Om,\beta)\cap \mathcal C(\bar{\Om});\\
(dd^cu)^m\we\beta^{n-m}=f\\
u|_{\partial\Om}=\varphi\end{cases}$$
has a smooth solution $u$.
\end{theorem}

Finally let us mention that convexity properties of the cones
$\Gamma_m$ yield the following mixed Hessian inequalities:
\begin{prop} Let $u_1,\ \cdots,\ u_m$ be m-sh $\mathcal C^2$ functions in some
domain $\Om\in\mathbb C^n$. Suppose
$(dd^cu_j)^m\we\beta^{n-m}=f_j$ for some continuous non negative
functions $f_j$. Then
$$dd^cu_1\we\cdots\we dd^cu_m\we\beta^{n-m}\geq(f_1\cdots f_m)^{\frac1m}
\beta^n.
$$
\end{prop}
\begin{proof} Pointwise this reduces to the G\aa rding inequality; see also
inequality 3. above for the case $u_2=u_3=\cdots=u_m$.
\end{proof}
Later on in Theorem \ref{weak} we shall see that the smoothness assumptions
here can be considerably relaxed.

\bigskip
\it K\"ahler setting. \rm   Given a compact K\"ahler manifold $
(X,\om)$ we can define the cones $\sh_m(X, \om)$ of those
functions u for which, in a local chart $\Om$ where $\om$ has a
potential $\rho$, the function $u+\rho$ belongs to
$\sh_m(\Om,\om)$. The definition is independent of the choice of
the chart and the potential. This essentially allows to carry over
all local reults to this setting. We refer to \cite{K2} for the
plurisubharmonic ($m=n$) case.
%and let the
%Reader to fill in the details in the general case.

The comparison principle on compact manifolds reads as follows:
\begin{prop}Let  $ (X,\om)$  be a compact K\"ahler manifold and $u, v$ be
 continuous functions in $\sh_m(X,\om)$. Then
$$\int_{\lbrace u<v\rbrace}(\om+dd^c v)^{m}\we\om^{n-m}
\leq\int_{\lbrace u<v\rbrace}(\om+dd^c u)^{m}\we\om^{n-m}.$$
\end{prop}
\begin{proof} One can repeat the proof for psh
functions from \cite{KI} or \cite{K2}.
\end{proof}
Observe that the cones $\Gamma_k(\om)$ are not fixed but according to an
observation of Hou (\cite{Hou}) these are invariant under the parallel transport defined by the Levi-Civita connection associated to $\om$.

\section{$L^{\infty}$ estimates and existence of weak solutions in domains}
In this section we state the results for $0<m<n.$ Let us denote by
$B(a,r)$ the ball in $\mathbb C ^n$ with center $a$ and radius
$r$. Let also $\om $ be a K\"ahler form defined in a neighbourhood
of the closure of a set $\Omega$ considered below and $V=\om ^n $
be the  volume form associated to $\omega$.

Let $ \mathcal{SH} _m (\om , \overline{\Om} )$ denote the class of
$m-\om$-sh functions which are continuous in $\overline{\Om}$.
\begin{prop}\label{volcapestimate} For $p<\frac{n}{n-m}$ and an open set
$\Om\subset
B(0,1)=B$ there exists $C(p)$ such that for any $ \
K\subset\subset \Om $,
$$
V (K) \leq C(p) cap _m^p  (K, \Om ),
$$
where
$$
cap _m (K, \Om )=\sup \{ \int _{K} (dd^c u)^m \we \om ^{n-m}, \
u\in \mathcal{SH} _m (\om , \overline{\Om} ), 0\leq u \leq 1 \}.
$$
\end{prop}
\begin{proof} If $V(K)=0$ then the inequality trivially holds. Assume from
now on that $V(K)>0$. Fix any $\ep \in (0,1/2)$ and set $f = [V (K)]^{2\ep -1}
\chi _K ,$ where $\chi _K$ denotes the characteristic function of
the set $K$.  Solve the complex \MAE\ in $B$ to find $v\in
PSH_{\om }(B)\cap C (\overline{\Om} )$ with $v=0$ on $\partial B$
and
$$
(dd^c v )^n =f\om ^n .
$$
By the  inequality between mixed Monge-Amp\`ere measures (see
\cite{K2}, \cite{D1})
\begin{equation}\label{kol1} (dd^c v)^m \we \om ^{n-m} \geq [V (K)]^{(2\ep
-1)\frac{m}{n} } \chi _K \om ^n .
\end{equation}
For $q=1+\ep$
$$
\int _B f^q dV = [V (K)]^{(2\ep -1)(1+\ep) +1} = [V (K)]^{\ep +
2\ep ^2} \leq V (B).
$$
So, by \cite{K0}, there exists $c>0$, independent of $K$ (though dependent on
$\ep$),  such
that $||v|| \leq 1/c .$ Take $u=cv$. Then, using (\ref{kol1})
$$
cap _m (K, \Om )\geq \int _{K} (dd^c u)^m \we \om ^{n-m} \geq c^m
[V (K)]^{(2\ep -1)\frac{m}{n} +1}.
$$
Therefore
$$
V (K) \leq C cap _m^{\frac{n}{n-m+2m\ep } } (K, \Om ),
$$
which proves the claim.

\end{proof}

\begin{prop}\label{lpbound} Let $\Om$ and $p$ be as above and
consider $u\in \mathcal{SH} _m (\om , \overline{\Om} )$ with $u=0
$ on $
\partial\Om$ and
$$
 \int _{\Om } (dd^c u)^m \we \om ^{n-m} \leq 1.
$$
Then for $U(s)= \{ u<-s \}$ we have
$$
cap _m (U(s), \Om )\leq s^{-m}
$$ and
$$
V_{2n} (U(s))\leq C(p) s^{-pm}.$$
 In particular $u\in L^q (\Om )$ for any $q<\frac{mn}{n-m},$ and
this remains true whenever $u$ is bounded in some
neighborhood of the boundary of $\Om .$

\end{prop}

\begin{proof} Fix $\ep >0, t>1$ and $K\subset U(s)$ and find $v\in
\mathcal{SH} _m (\om , \overline{\Om} )$ with $-1\leq v\leq 0$ and
$$
\int _{K } (dd^c v)^m \we \om ^{n-m} \geq cap _m (K, \Om ) -\ep .
$$
Then, using the comparison principle \cite{BT}, \cite{Bl}
$$\aligned
& cap _m (K, \Om ) -\ep \leq \int _{K } (dd^c v)^m \we \om
^{n-m} \\
\leq & \int _{\{-\frac{t}{s} u <v \} } (dd^c v)^m \we \om ^{n-m}
\leq (\frac{t}{s})^m \int _{\Om } (dd^c u)^m \we \om ^{n-m} \leq
(\frac{t}{s})^m .
\endaligned
$$

To finish the proof of the first estimate recall that $cap _m
(U(s), \Om )$ is the supremum of $cap _m (K, \Om )$ over compact
$K\subset U(s)$ and  let $\ep \to 0$ and $t\to 1$. Then the
estimate of the volume follows from Proposition \ref{volcapestimate}.

\end{proof}
\begin{remark}\label{ex} The bound for $q$ above is optimal as the function
$$
G(z) = -|z|^{2-2n/m}
$$
is m-sh and belongs to $L^q _{loc}$ if and only if
$q<\frac{mn}{n-m}.$
\end{remark}

In \cite{Bl} B\l ocki conjuctured that any $m$-sh function belongs
to $ L^q_{loc}  (\Om )$ for any $q<\frac{mn}{n-m}.$ He proved this
for $q<\frac{n}{n-m}.$ The above proposition confirms partially
the conjecture - under the extra assumption of boudedness near the
boundary. Still the question about the local inegrability remains
open.

We now proceed to proving the $L^{\infty} $ a priori estimates for
the Hessian equation with the right hand side controlled in terms
of the capacity.

 \begin{lemma}\label{kollemma} For $p\in (1, \frac{n}{n-m})$ and an
open set $\Om \subset B$ consider $u,v \in  \mathcal{SH} _m (\om ,
\overline{\Om}  )$ satisfying
$$
\int _{K } (dd^c u)^m \we \om ^{n-m} \leq Acap _m^p (K, \Om )
$$
for some $A>0$ and any compact $K\subset \Om$. If the sets
$U(s)=\{u-s <v\}$ are nonempty and relatively compact in $\Om$ for
$s\in (s_0 , s_0 +t_0 )$ then there exists a constant $C(p,A)$
such that
$$
t_0 \leq C(p,A) cap _m^{p/n} (U(s_0 +t_0 ), \Om ) .
$$
\end{lemma}

\begin{proof}  Using the notation
$$
a(s)=cap _m (U(s), \Om ), \ \ \ b(s) =\int _{U(s) } (dd^c u)^m \we
\om ^{n-m}
$$
we claim that
\begin{equation}\label{2}
t^m a(s)\leq b(s+t) , \ \ \ t\in (0,  s_0 +t_0 -s).
\end{equation}
Indeed, for fixed compact $K\subset U(s)$ take $w_1 \in
\mathcal{SH} _m (\om , \overline{\Om}  ), \  -1\leq w_1 \leq 0$
such that
$$
\int _{K } (dd^c w_1 )^m \we \om ^{n-m} \geq cap _m (K, \Om ) -\ep
.
$$
Then for $w_2 =\frac{1}{t} (u-s-t)$ one readily verifies that
$K\subset V\subset U(s+t)$, where $V= \{ w_2 <w_1 + \frac{1}{t} v
\}.$ So, by the comparison principle

$$\aligned
& cap _m (K, \Om ) -\ep \leq \int _{K } (dd^c
(w_1+\frac{1}{t}v))^m \we \om
^{n-m} \\
\leq & \int _{V } (dd^c (w_1+\frac{1}{t}v))^m \we \om ^{n-m} \leq
\int _{V } (dd^c w_2 )^m \we \om ^{n-m} \leq t^{-m}b(s+t) .
\endaligned
$$
Having (\ref{2}) one proceeds as in the proof of Lemma 4.3 in
\cite{K1} (with $h(x) =x^{m(p-1)}$) to reach the conclusion.

\end{proof}

Coupling this with the volume estimate in Proposition
\ref{volcapestimate} we obtain a priori estimates for the
solutions of Hessian equations with the  right hand side in some
$L^q$ spaces.

 \begin{theorem}\label{stability} Take $q>n/m .$ Then the conjugate $q'$ of $q$
satisfies $q' <n/(n-m)$. Fix $p'\in (q' , n/(n-m))$ and $p=p'/q'
>1.$ Consider $u, v \in \mathcal{SH} _m (\om , \overline{\Om })$ such that $u\geq v $
on $
\partial\Om $,   $\{ u<v\} \neq \emptyset $ and
$$
 (dd^c u)^m \we \om ^{n-m} = f\om ^n
$$
for some $f\in L^q (\Om , dV).$ Then
$$
\sup (v-u) \leq c(p',q, ||f||_{L^q(\Om )} )||(v-u)_+
||_{L^{q'}(\Om )}^{\frac{p}{n+p(m+1)} } , \ \ (v-u)_+ :=\max (v-u,
0).
$$
\end{theorem}
\begin{proof}
By the H\"older inequality and Proposition \ref{volcapestimate}, for a
compact set
$K\subset \Om$ we have
$$
\int _K f\om ^n \leq ||f||_{q} V(K)^{1/q' } \leq
C(p)||f||_{L^q(\Om )}cap _m^p (K, \Om ).
$$
Therefore, by Lemma \ref{kollemma}, we get for $t=\frac{1}{2} \sup (v-u)$ and
$E(t)= \{ u+t< v \}$
\begin{equation}\label{3ab}
t\leq c(p',q, ||f||_{L^q(\Om )} ) cap _m^{p/n} (E(t), \Om ).
\end{equation}

To shorten notation set $a(t) =cap _m (E(2t), \Om ).$ Take  $w \in
\mathcal {SH} _m (\om , \overline{\Om} ), \  -1\leq w \leq 0$ such
that
$$
\int _{E(2t) } (dd^c w )^m \we \om ^{n-m} \geq \frac{1}{2} a(t).
$$
Observe that for $V=\{ u < tw +v-t \}$ the following inclusions
hold
$$
E(2t) \subset V\subset E (t ).
$$
Applying the comparison principle  we thus get
$$\aligned
\frac{1}{2}a(t) t^m &\leq \int _{E(2t) } [dd^c  (tw +v)]^m \om
^{n-m}\leq
\int _V  (dd^c u )^m \we \om ^{n-m}\\
&\leq\int_{E (t ) }f \, dV.
\endaligned
$$
Hence from the H\"older inequality one infers
$$
a(t) t^{m+1} \leq 2 \int_{\Om }(v-u)_+ f \, dV   \leq
||f||_{L^q(\Om )} ||(v-u)_+ ||_{q'}.
$$
Inserting this estimate into (\ref{3ab}) we arrive at
$$
t\leq c_1 (p',q, ||f||_{L^q(\Om )} ) [||f||_{q} ||(v-u)_+
||_{L^{q'}(\Om )} t^{-m-1}]^{p/n}
$$
and consequently
$$
t\leq  c_2 (p',q, ||f||_{L^q(\Om )} ) ||(v-u)_+ ||_{L^{q'}(\Om )}
^{\frac{p}{n+p(m+1)} } .
$$
\end{proof}
\begin{corollary}\label{linfty}
The last theorem gives a priori $L^{\infty }$ estimate for the
solutions of the Hessian equation (\ref{hes2}) with the right hand
side in $L^q$ and a fixed boundary condition.
\end{corollary}
Indeed, we apply the theorem for the solution $u$ of $$
 (dd^c u)^m \we \om ^{n-m} = f\om ^n
$$
with given continuous boundary data $\varphi$ and for $v$, which
is the maximal function in $\mathcal{SH} _m (\om , \overline{\Om}
)$ matching the boundary condition (it exists by \cite{Bl}). Then
$u$ is bounded by a constant depending on $\Om , ||\varphi
||=||v||, $ and $||f||_q$ since $||(v-u)_+ ||_{L^{q'}(\Om )}$ is
bounded (Proposition \ref{lpbound}).

\begin{corollary}\label{eq}  The solutions of  the Hessian equation with the
right hand sides uniformly bounded in $L^q$  $q>n/m $ and given
continuous boundary data form an equicontinuous family.
\end{corollary}
For the proof follow \cite{K2} p. 35, which deals with the
Monge-Amp\`ere case.

\bigskip
Below we state yet another stability theorem which we shall need
later. Given the estimates we have already proven its proof
follows the arguments  from \cite{K0}.
\begin{theorem}\label{stabilitybis}
 Let $q>n/m $. Consider $u, v\in \mathcal{SH} _m (\om , \overline{\Om} )$ such that
   $\{ u<v\} \neq \emptyset $ and
$$
 (dd^c u)^m \we \om ^{n-m} = f\om ^n,\  (\dc v)^m\we\om^{n-m}=g\om^n
$$
for some $f, g\in L^q (\Om , dV).$ Then
$$
\sup_\Om (v-u) \leq sup_{\partial{\Om}}(v-u)+ c(q,m,n,diam(\Om)
)||f-g||_{L^{q}(\Om )}^{1/m}.
$$
\end{theorem}
\begin{remark} The analogous stability theorem for the real m-Hessian equation
($m<n/2$) can be found in \cite{Wa2}, Theorem 5.5 (see also \cite{CW}). There the
 optimal exponent $q$  is equal to $n/2m$.
\end{remark}

Next we obtain a theorem on the existence of weak, continuous
solutions when $\om =\beta$ and the right hand side is in $L^q , \
q>n/m .$
\begin{theorem}\label{weaksol} Let $\Om$ be smoothly bounded (m-1)-pseudoconvex
domain (as in Theorem  \ref{liexist}). Then for $q>n/m $, $f\in
L^q (\Om , dV)$ and continuous $\varphi$ on $\partial\Om$ there
exists $u\in \mathcal{SH} _m (\om , \overline{\Om} )$ satisfying
$$
 (dd^c u)^m \we \beta ^{n-m} = f\beta ^n
$$
and $u=\varphi$ on $\partial\Om$.
\end{theorem}

\begin{proof} For smooth, positive $f$ this is the result of Li
\cite{Li} (Theorem \ref{liexist}). With our assumptions we
approximate $f$ in $L^q (\Om , dV)$ by smooth positive $f_j$ and
approximate uniformly $\varphi$ by smooth $\varphi _j$ . The
solutions $u_j$ corresponding to $f_j , \varphi _j$ are
equicontinuous and uniformly bounded (Corollaries \ref{linfty},
\ref{eq}). Thus we can pick up a subsequence converging uniformly
to some $u\in \mathcal{SH} _m (\om , \overline{\Om} )$. By the
convergence theorem  $u$ solves the equation.

\end{proof}
\begin{remark} Observe that for $\om=\beta$, the plurisubharmonic
function $u(z)=log||z||$ has a $m$-Hessian density in $L^p$ for
any $p<n/m$ which shows that the exponent $n/m$ is optimal.
\end{remark}

Equipped with the existence and stability of weak solutions we can
also prove the  weak G\aa rding inequality announced in Section 1:
\begin{theorem}\label{weak}
 Let $u_1,\ \cdots,\ u_m$ be locally bounded m-sh  functions in some
domain $\Om\in\mathbb C^n$. Suppose $(dd^cu_j)^m\we\beta^{n-m}=
f_j\beta^n$ for some nonnegative functions $f_j\in L^q(\Om),\
q>n/m$. Then
$$dd^cu_1\we\we\cdots\we dd^cu_m\we\beta^{n-m}\geq(f_1\cdots f_m)^{\frac1m}
\beta^n.
$$
\end{theorem}
\begin{proof} We can essentially follow the lines of the proof of the
analogous result for psh functions from \cite{KI} (see also
\cite{K2}). First observe that the inequality is purely local
hence it suffices to prove it under the additional assumptions
that $\Om$ is a ball and all the functions $u_i$ are defined in a
slightly bigger ball. Hence one can use  convolutions with
smoothing kernel to produce a decreasing to $u_i$ sequence of m-sh
functions $\lbrace u_{i,j}\rbrace_{j=1}^{\infty}$ (cf. Proposition
\ref{aaaa}). Then given any collection of smooth  positive
functions $ f_{i,k}\in L^q(\Om),\ q> n/m$ by \cite{Li} we can
solve  the Dirichlet problems
$$\begin{cases}v_{i,j,k}\in \sh_m(\Om)\cap\mathcal C^{\infty}(\Om)\\
(dd^cv_{i,j,k})^m\we\beta^{n-m}=f_{i,k}\beta^n\\
v_{i,j,k}|_{\partial\Om}=u_{i,j}.
\end{cases}$$
For those smooth functions  we can apply pointwise the G\aa rding
inequality to conclude that
$$dd^cv_{1,j,k}\we\cdots\we dd^cv_{m,j,k}\we\beta^{n-m}\geq(f_{1,k}
\cdots f_{m,k})^{\frac1m} \beta^n$$ for any $j,\ k\geq 1$. Then
given any non negative $f_{i}\in L^q(\Om),\ q> n/m$ we can find an
approximating sequence of smooth positive $\lbrace
f_{i,k}\rbrace_{k=1}^{\infty}$ which converge in $ L^q$ to $f_i$.
By the stability theorem the corresponding solutions $v_{i,j,k}$
(recall they the same boundary values $u_{i,j}$) converge
uniformly as $k\rightarrow\infty$ to the m-sh functions $v_{i,j}$
(solving the limiting weak equation), and hence the inequality
follows from the continuity of Hessian currents under uniform
convergence of their potentials. Now if we let
$j\rightarrow\infty$ the boundary vaules decrease towards $u_i$
and hence so do the functions $v_{i,j}$ by the comparison
principle. The convergence is not uniform but monotonicity is
still sufficient to guarantee the continuity and hence in the
limit we obtain the claimed inequality.
\end{proof}
\begin{remark} The weak G\aa rding inequality can be further generalized similarly to the $m=n$ case as in \cite{D1}.
\end{remark}
\section{$L^{\infty}$ estimates and existence of weak solutions
on compact K\"ahler manifolds}

The a priori estimates from the previous section can be carried
over to the case of compact K\"ahler manifolds as it was done in
\cite{KI} or \cite{K2} for the \MAE . Let us consider a compact
n-dimensional K\"ahler manifold $X$ equipped with the fundamental
form $\om$ and  recall that a continuous function $u$ is
$\om-m$-subharmonic (shortly: $\om-m$-sh) on $X$ if
$$
(\om +dd^c u)^k \we \om ^{n-k} \geq 0, \ \ k=1, 2, ... , m .
$$
The set of such functions is denoted by $\mathcal{SH} _m (X, \om
).$ We study the complex  m-Hessian equation
\begin{equation}\label{3} (\om +dd^c u)^m \we \om ^{n-m} = f\om
^n
\end{equation}
with given nonnegative function $f\in L^1 (M) $, which is
normalized by the condition
$$
\int_X f\, \om ^n  =\int_X \om ^n .
$$
The solution is required to be $\om-m$-sh. By the result of Hou
\cite{Hou} the solutions of the equation, for smooth positive $f$,
exist on manifolds with nonnegative holomorphic bisectional
curvature. In that case our a priori estimates will also give the
existence of weak solutions for $f\geq 0$ in $L^q ,\ q>n/m.$

We define for a compact set $K\subset X$  its capacity
$$
cap _m (K)=\sup \{ \int _K (\om +dd^c u)^m \we \om ^{n-m}  : u \in
\mathcal{SH} _m (X, \om ), 0\leq u \leq 1 \}.
$$
To use the local results we need also a capacity defined as
follows. Let us consider two finite coverings by strictly
pseudoconvex sets $\{ B_s \} ,\ \{ B'_s \} , \ s=1,2, ... ,N$ of
$X$ such that $\bar{B}'_s \subset B_s $ and in each $B_s$ there exists
$v_s \in PSH (B_s )$ with $dd^c
v_s =\om $ and $v_s =0$ on $\partial B_s .$ Given a compact set
$K\subset X$ define $K_s =K \cap \overline{B'_s}$. Set
$$cap' _m(K)
= \sum _s cap (K_s ,B_s ),$$
 where $cap_m (K,B)$  denotes the relative capacity from the previous section.
As in \cite{KI} one can show that $cap _m (K)$ is comparable with
$cap' _m (K)$: There exists $C>0$ such that
$$
\frac{1}{C} cap _m (K) \leq cap' _m (K) \leq C cap _m (K).
$$
Hence, by Proposition \ref{volcapestimate} we have
$$
V (K) \leq C(p,X) cap _m^p  (K),
$$
for $p<\frac{n}{n-m}$ and $V$ the volume measured by $\om ^n$.

With this estimate at our disposal we can obtain the same a priori
estimates as in domains in $\mathbb C^n$. The proofs are almost
identical. In the compact setting one has to make sure that
instead of just a sum of $m$-sh functions one considers a convex
combination of $\om-m$-sh functions (see \cite{K2}). In particular
the following  theorems hold.

\begin{theorem}\label{stabilityK} Consider $q>n/m $, its  conjugate
$q'$ and  $p'\in (q' , n/(n-m))$. Write $p=p'/q'
>1.$ Consider $u,v \in \mathcal{SH} _m (X, \om )$ such that  $\{ u<v\} \neq
\emptyset $ and
$$
 (\om + dd^c u)^m \we \om ^{n-m} = f\om ^n
$$
for some $f\in L^q (dV).$ Then
$$
\sup (v-u) \leq c(p',q, ||f||_{L^q (X)} )||(v-u)_+
||_{q'}^{\frac{p}{n+p(m+1)} } , \ \ (v-u)_+ :=\max (v-u, 0).
$$
\end{theorem}

\begin{corollary}\label{eqK}  The family of solutions of  the Hessian equation
(\ref{3}) with the
right hand sides uniformly bounded in $L^q$  $q>n/m $ are
equicontinuous.
\end{corollary}

Applying the theorem of Hou \cite{Hou} and the above statements
one immediately gets the following existence theorem.

\begin{theorem}\label{weaksolK} Let $X$ be a compact K\"ahler manifold
 with nonnegative holomorphic bisectional
curvature. Then for $q>n/m $ and $f\in L^q (dV)$  there exists a
unique function $u\in \mathcal{SH} _m (X, \om )$ satisfying
$$
 (\om + dd^c u)^m \we \om ^{n-m} = f\om ^n
$$
and $\max u =0.$
\end{theorem}

\section{Local regularity}

In this section we prove a counterpart of the main result in B\l
ocki-Dinew [BD], where the case of the \MAE\  was studied. We
shall treat only the $\om =\beta $ case and use PDE notation (with
$\sigma_m$ defined in Introduction).

\begin{theorem}\label{lore} Assume that $n\geq 2$ and $p>n(m-1)$. Let $u\in
W^{2,p}(\Omega)$, where $\Omega$ is a domain in $\mathbb C^n$, be a
m-subharmonic solution of
  \begin{equation}\label{equation4}\sigma_m\big(u_{z_j\bar z_k}\big)=\psi>0.
\end{equation}
Assume that $\psi\in C^{1,1}(\Omega)$. Then for
$\Omega'\Subset\Omega$
  $$\sup_{\Omega'}\Delta u\leq C,$$
where $C$ is a constant depending only on $n$, $m$, $p$, $\text{\rm
dist}(\Omega',\partial\Omega)$, $\inf_\Omega\psi$, $\sup_\Omega\psi$,
$||\psi||_{C^{1,1}(\Omega)}$ and $||\Delta u||_{L^p(\Omega)}$.
\end{theorem}
\begin{proof}
By $C_1,C_2,\dots$ we will denote possibly different constants
depending only on the required quantities. Without loss of
generality we may assume that $\Omega=B$ is the unit ball in $\mathbb
C^n$ and that $u$ is defined in some neighborhood of $\bar B$. We
will use the notation $u_j=u_{z_j}$, $u_{\bar j}=u_{\bar z_j}$ with the notable
exception of $u_{(\ep)}$ which is defined below.

Let us define, following \cite{BT}, the Laplacian approximating
operator
  $$T=T_\ep(u)=\frac{n+1}{\ep^2}(u_{(\ep)}-u),$$
where
  $$u_{(\ep)}(z)=\frac 1{V(B(z,\ep))}\int_{B(z,\ep)}u\,dV.$$
 Since
$T_\ep u\to\Delta u$ weakly as $\ep\to 0$, it is enough to show a
uniform upper bound for $T$ independent of $\ep$. Observe that since $u$ is subharmonic we have $T_\ep(u)\geq 0$.

Before we continue let us state two  lemmas. The first one is
classical.
\begin{lemma}\label{second} Let $u\in
W^{2,p}(\Omega)$ ( $\Omega$ is a domain in $\mathbb C^n$) be a subharmonic
function. Given any $\Omega'\Subset\Omega$
the operator $T_\ep(z)$ is well defined on  $\Omega'$ for any sufficiently
small $\ep>0$. Furthermore
$$||T_\ep||_{L^p(\Omega')}\rightarrow ||\Delta u||_{L^p(\Omega')}$$
in particular $||T_\ep||_{L^p(\Omega')}$ is uniformly bounded for all
$0<\ep<\ep_0$.
\end{lemma}

\begin{lemma}\label{third} The function $T_\ep(u)(z)$ for any $\ep>0$ satisfies the following
subharmonicity condition:
$$\frac{\partial\sigma_m\big(u_{j\bar k}\big)}{\partial u_{i\bar{j}}}
T_{\ep,i\bar{j}}\geq  -C_1,$$
where $\frac{\partial\sigma_m\big(u_{j\bar k}\big)}{\partial u_{i\bar{j}}}$
is the $(i,j)$-th (m-1)-cominor of the matrix $u_{i\bar{j}}(z)$
 and $C_0$ is a constant dependent only on $n$, $m$,  $\inf_\Omega\psi$,
 $\sup_\Omega\psi$, and
$||\psi||_{C^{1,1}(\Omega)}$ .
\end{lemma}
\begin{proof}
Observe that $u_{(\ep)}$ is a convex combination of $m$-subharmonic
functions, hence it is m-subharmonic. Therefore one has the
inequality
$$(\dc u_{(\ep)})^{m}\we\om^{n-m}\geq 0.$$
In fact following the lines of the same argument in \cite{BT}
(where it was applied to the Monge-Amp\`ere operator) one can
prove the stronger inequality
\begin{equation}\label{convolutionhess}
(\dc u_{(\ep)})^{m}\we\om^{n-m}\geq ((\psi^{1/m})_{(\ep)})^m.
\end{equation}
Indeed, for smooth $u$ this is just a consequence of the concavity of $\sigma_m^{1/m}$. For nonsmooth solutions one can repeat the Goffman-Serrin formalism just as in \cite{BT}.

Thus using the weak G\aa rding inequality (Theorem \ref{weak}) one
has
\begin{equation*}
(\dc u)^{m-1}\we\dc u_{(\ep)}\we\om^{n-m}\geq \psi^{(m-1)/m} (\psi^{1/m})_{(\ep)} dV.
\end{equation*}
Next, identifying $(n,n)$ forms and their densities one gets, up
to a multiplicative numerical constant $c_{n,m},$ the following
string of inequalities
\begin{align*}
&\frac{\partial\sigma_m\big(u_{j\bar k}\big)}{\partial
u_{i\bar{j}}}
T_{\ep,i\bar{j}}=c_{n,m}1/\ep^2\dc( u_{(\ep)}-u)\we(\dc u)^{m-1}\we\om^{n-m}\\
\geq &c_{n,m}1/\ep^2
\psi^{(m-1)/m}((\psi^{1/m})_{(\ep)}-\psi^{1/m})=c_{n,m} \psi^{(m-1)/m}
T_\ep(\psi^{1/m}).
\end{align*}
But $\psi$ is a strictly positive $\mathcal C^{1,1}$ function hence
$T_\ep(\psi^{1/m})\geq-C_1(||\psi||, ||\psi^{1/m}||_{C^{1,1}})$. Combining
all those inequalities we obtain the claimed estimate.
\end{proof}

>From now on we drop the indice $\ep$ in what follows. We will use
the same calculations as in \cite{BD} which in turn relied on
\cite{Tr2}. For some $\alpha,\beta\geq 2$ to be determined later
set
  $$w:=\eta(T)^\alpha,$$
where
  $$\eta(z):=(1-|z|^2)^\beta .$$
Then
 $$w_i=\eta_i(T)^\alpha
     +\alpha\eta(T)^{\alpha-1}(T)_i$$
and
  \begin{align*} \frac{\partial\sigma_m\big(u_{j\bar k}\big)}{\partial
  u_{i\bar{j}}}w_{i\bar j}
     &=\alpha\eta(T)^{\alpha-1}\frac{\partial\sigma_m\big(u_{j\bar k}
     \big)}{\partial u_{i\bar{j}}}(T)_{i\bar j}
        +\alpha(\alpha-1)\eta(T)^{\alpha-2}
              \frac{\partial\sigma_m\big(u_{j\bar k}\big)}{\partial
              u_{i\bar{j}}}(T)_i(T)_{\bar j}\\
     &\ \ \ \ \ \ \ \ +2\alpha(T)^{\alpha-1}\text{Re}
               \big(\frac{\partial\sigma_m\big(u_{j\bar k}\big)}
               {\partial u_{i\bar{j}}}\eta_i(T)_{\bar j}\big)
         +(T)^\alpha \frac{\partial\sigma_m\big(u_{j\bar k}\big)}
         {\partial u_{i\bar{j}}}\eta_{i\bar j}.
\end{align*}
By Lemma \ref{third}   and the Schwarz inequality
for $t>0$
 \begin{align*} &\frac{\partial\sigma_m\big(u_{j\bar k}\big)}
 {\partial u_{i\bar{j}}}w_{i\bar j}
     \geq-C_1\alpha\eta(T)^{\alpha-1}
        +\alpha(\alpha-1)\eta(T)^{\alpha-2}
            \frac{\partial\sigma_m\big(u_{j\bar k}\big)}
 {\partial u_{i\bar{j}}}(T)_i(T)_{\bar j}\\
     & -t\alpha(T)^{\alpha-1}
          \frac{\partial\sigma_m\big(u_{j\bar k}\big)}
 {\partial u_{i\bar{j}}}(T)_i(T)_{\bar j}
        -\frac 1t\alpha(T)^{\alpha-1}\frac{\partial\sigma_m\big(u_{j\bar k}
        \big)}{\partial u_{i\bar{j}}}\eta_i\eta_{\bar j}
         +(T)^\alpha \frac{\partial\sigma_m\big(u_{j\bar k}\big)}
         {\partial u_{i\bar{j}}}\eta_{i\bar j}.
\end{align*}
Therefore with $t=(\alpha-1)\eta/T$ we get
  $$\frac{\partial\sigma_m\big(u_{j\bar k}\big)}{\partial
  u_{i\bar{j}}}w_{i\bar j}
      \geq-C_1\alpha\eta(T)^{\alpha-1}
         +(T)^\alpha \frac{\partial\sigma_m\big(u_{j\bar k}\big)}
         {\partial u_{i\bar{j}}}\left(\eta_{i\bar j}
            -\frac\alpha{\alpha-1}\frac{\eta_i\eta_{\bar j}}\eta\right).$$
We now have
 \begin{align*}\eta_i&=-\beta z_i\eta^{1-1/\beta}\\
      \eta_{i\bar j}&=-\beta\delta_{i\bar j}\eta^{1-1/\beta}
         +\beta(\beta-1)\bar z_iz_j\eta^{1-2/\beta},
\end{align*}
and thus  $$\big|\eta_{i\bar j}\big|,\ \big|\frac{\eta_i\eta_{\bar j}}
\eta\big|\leq
      C(\beta)\,\eta^{1-2/\beta}.$$
Coupling the above inequalities we get
  $$\frac{\partial\sigma_m\big(u_{j\bar k}\big)}{\partial
  u_{i\bar{j}}}w_{i\bar j}\geq -C_2(T)^{\alpha-1}
         -C_3w^{1-2/\beta}(T)^{2\alpha/\beta}\sum_{i,j}|\frac{\partial
         \sigma_m\big(u_{j\bar k}\big)}{\partial u_{i\bar{j}}}|.$$
Fix $q$ with $n/m<q<p/m(m-1)$ (by our assumption on $p$ such a
choice is possible). By Lemma \ref{second} $||T||_p$ and $||\Delta
u||_p$ are under control. By Calderon-Zygmund inequalities we
control $||u_{i\bar j}||_p$ too. Observe that
$\frac{\partial\sigma_m\big(u_{j\bar k}\big)}{\partial
u_{i\bar{j}}}$ is a sum of of products of $m-1$ factors of the type
 $u_{i\bar j}$ and therefore $||\frac{\partial\sigma_m\big(u_{j\bar k}\big)}
 {\partial u_{i\bar{j}}}||_{p/(m-1)}$ is also under control.
It follows that for
 $$\alpha=1+\frac p{qm},\ \ \ \beta=2\big(\frac{qm+p}{p-qm(m-1)}\big)$$
we have
  $$||(\frac{\partial\sigma_m\big(u_{j\bar k}\big)}{\partial
  u_{i\bar{j}}}w_{i\bar j})_-||_{qm}
       \leq C_3(1+(\sup_B\,w)^{1-2/\beta}),$$
where $f_-:=-\min(f,0)$.
By Theorem \ref{weaksol} we can find continuous $m$-subharmonic $v$ vanishing
on
$\partial B$ and such that
  $$\sigma_m(v_{i\bar j})=((u^{i\bar j}w_{i\bar j})_-)^m.$$
Then the weak G\aa rding inequality yields
\begin{align*} &\frac{\partial\sigma_m\big(u_{j\bar k}\big)}
 {\partial u_{i\bar{j}}}v_{i\bar j}
     =c_{n,m}(\dc u)^{m-1}\we\dc v\we\om^{n-m}\\ \geq & c_{n,m}
     (\sigma_m(u_{i\bar j})^{(m-1)/m}(\sigma_m(v_{i\bar j}))^{1/m}
     \geq 1/C_4( \frac{\partial\sigma_m\big(u_{j\bar k}\big)}
     {\partial u_{i\bar{j}}}w_{i\bar j})_-\\
    & \geq-\frac 1{C_4} \frac{\partial\sigma_m\big(u_{j\bar k}\big)}
     {\partial u_{i\bar{j}}}w_{i\bar j}.
\end{align*}
By maximum prinicple we obtain that  $w\leq -C_4v$, since this
inequality holds on $\partial B$. Applying the stability theorem
(Theorem \ref{stabilitybis}), with $u=0$, we get
 \begin{align*}\sup_Bw\leq C_4||v||&\leq C_5(||\sigma_m(v_{i\bar j})||_q^{1/m})
     =C_5||( \frac{\partial\sigma_m\big(u_{z_j\bar z_k}\big)}{\partial
     u_{i\bar{j}}}w_{i\bar j})_-||_{qn}\\
     &\leq C_6(1+(\sup_B\,w)^{1-2/\beta}).
\end{align*}
Therefore $w\leq C_7$ and thus  $$T^{\alpha }\leq\frac
{C_7}{\eta}$$ which is the desired bound.

\end{proof}
\begin{remark} The analogous reasoning can be applied also to the real m-Hessian
equation (using Wang stability theorem and existence of weak solutions). It
 turns out that for $m<n/2$ the corresponding exponent in the $W^{2,p}$ Sobolev
 space is equal to $n(m-1)/2$. Observe that this improves the $m(n-1)/2$ exponent
obtained by different methods by Urbas \cite{U}. Whether this
exponent is optimal is however still unclear and would require
construction of suitable Pogorelov type Hessian examples.
\end{remark}

\noindent {\bf Acknowledgements}.  The  authors were
partially supported by NCN grant 2011/01/B/ST1/00879.  The first
author was also supported by Polish ministerial grant ``Iuventus Plus''
and Kuratowski fellowship granted by the Polish Mathematical
Society (PTM) and Polish Academy of Science (PAN).

{ Rutgers University, Newark, NJ 07102, USA;\\  Faculty of Mathematics and Computer Science,  Jagiellonian University 30-348 Krakow, Lojasiewicza 6, Poland;\\ e-mail: {\tt slawomir.dinew@im.uj.edu.pl}}\\ \\

{\noindent Faculty of Mathematics and Computer Science, Jagiellonian University 30-348 Krakow, Lojasiewicza 6, Poland;\\ e-mail: {\tt slawomir.kolodziej@im.uj.edu.pl}}\\ \\

\end{document}